\documentclass{article}
\usepackage[utf8]{inputenc}
\usepackage[utf8]{inputenc}
\usepackage{biblatex}
\usepackage{tikz}
\usepackage{graphicx}
\usepackage{caption}
\usepackage{subcaption}
\usepackage{lineno,hyperref}
\usepackage{amssymb, amsthm, amsmath}
\usepackage{mathtools}
\usepackage{graphics} %library for inserting images
\usepackage{dsfont}
\usepackage{amsthm}
\usepackage[english]{babel}
\usepackage{fancyhdr}
\usepackage{wrapfig}
\usepackage{color}
\usepackage[shortlabels]{enumitem}
\usepackage{hyperref}
\usepackage[utf8]{inputenc}
\usepackage{multirow}
\usepackage{comment}

\theoremstyle{plain}

\newtheorem{corollary}{Corollary}[section]
\newtheorem{theorem}{Theorem}[section]
\newtheorem{lemma}{Lemma}[section]

\theoremstyle{definition}
\newtheorem{definition}{Definition}[section]

\usepackage{authblk}
\author[1]{Aidan Johnson}
\author[2]{Andrew E. Vick}
\author[3]{Darren A. Narayan}

\affil[1]{University of Minnesota}
\affil[2]{Lee University}
\affil[3]{Rochester Institute of Technology}

\title{Characterization of Graphs Failed Skew Zero Forcing Number of 1}

\date{July 2022}

\begin{document}

\maketitle

\begin{abstract}
  Given a graph $G$, the zero forcing number of $G$, $Z(G)$, is the smallest cardinality of any set $S$ of vertices on which repeated applications of the forcing rule results in all vertices being in $S$. The forcing rule is: if a vertex $v$ is in $S$, and exactly one neighbor $u$ of $v$ is not in $S$, then $u$ is added to $S$ in the next iteration. Hence the failed zero forcing number of a graph was defined to be the size of the largest set of vertices which fails to force all vertices in the graph. A similar property called skew zero forcing was defined so that if there is exactly one neighbor $u$ of $v$ is not in $S$, then $u$ is added to $S$ in the next iteration. The difference is that vertices that are not in $S$ can force other vertices. This leads to the failed skew zero forcing number of a graph, which is denoted by $F^{-}(G)$. In this paper we provide a complete characterization of all graphs with $F^{-}(G)=1$. Fetcie, Jacob, and Saavedra showed that the only graphs with a failed zero forcing number of $1$ are either: the union of two isolated vertices; $P_3$; $K_3$; or $K_4$. In this paper we provide a surprising result: changing the forcing rule to a skew-forcing rule results in an infinite number of graphs with $F^{-}(G)=1$. 
\end{abstract}

\section{Introduction}
Given a graph $G$, the zero forcing number of $G$, $Z(G)$, is the smallest cardinality of any set $S$ of vertices on which repeated applications of the forcing rule results in all vertices being in $S$. The forcing rule is: if a vertex $v$ is in $S$, and exactly one neighbor $u$ of $v$ is not in $S$, then $u$ is added to $S$ in the next iteration. Zero forcing numbers have attracted great interest over the past 15 years and have been well studied \cite{Fallat}. Investigations of the largest size of a set $S$ that does not force all of the vertices in a graph to be in $S$. This quantity is known as the failed zero forcing number of a graph and is denoted by $F(G)$ \cite{Fetcie} and \cite{Adams}. Shitov \cite{Shitov}, proved that determining the failed zero forcing number of a graph is NP-complete. Independently, a closely related property called the zero blocking number of a graph was introduced in 2020 by Beaudouin-Lafona, Crawford, Chen, Karst, Nielsen, and Sakai Troxell \cite{Beaudouin} and Karst, Shen, and Vu \cite{KarstB}. The zero blocking number of a graph $G$ equals $|V(G)|-F(G)$. 
In 2010, researchers from the IMA-ISU research group \cite{ISU} introduced skew zero forcing where any vertex that has all but one of its neighbors colored will force the last remaining vertex to be forced.  
In 2016, Ansill, Jacob, Penzellna, and Saavedra \cite{Ansill} introduced the failed skew zero forcing number, which is the largest size of a set of vertices that does not skew force all of the vertices in the graph.

We will use $K_n$, $P_n$, $C_n$, to denote the complete graph, path, and cycle on $n$ vertices, respectively. A vertex $v$ is a cut-vertex of a graph $G$ if $G-v$ has more components than $G$. 

Throughout the paper we will use colorings to describe failed zero forcing sets where the vertex in $S$ will be referred to as colored and the vertices not in $S$ will be referred to as uncolored.

In 2016 Ansill, Jacob, Penzellna, and Saavedra \cite{Ansill} characterized graphs with extreme values of $F^{-}(G)=0,n-2,n-1$ and $n$. The characterization of graphs with $F^{-}(G)=0$ was surprisingly complex. We review this in the next section.

In this paper we characterize all graphs with a failed skew zero forcing number of $1$. In these graphs, there is a vertex which is a stalled set (meaning that this set of vertices does not force any other vertices), and any pair of vertices forces all vertices in the graph. Within this class of graphs there is an interesting subclass of graphs where there is a unique vertex $v$ that is a failed skew zero forcing set of size 1 and all other vertices in the graph are zero forcing sets.

In 2015, Fetcie, Jacob, and Saavedra \cite{Fetcie}, characterized all graphs with a failed zero forcing number of $1$ - which turned out to be only four graphs: a pair of isolated vertices; $K_3$; $P_3$; and $P_4$. It is surprising to see that the change of the forcing rule where vertices not in $S$ can also force other vertices results in an infinite number of graphs with a failed skew zero forcing number of $1$.

\section{Failed skew zero forcing numbers of graphs}

Ansill, Jacob, Penzellna, and Saavedra \cite{Ansill} presented new results involving failed skew zero forcing numbers of graphs. In particular they provided a characterization for graphs with failed skew zero forcing numbers of $0$. 

They described a family of graphs called doubly extended bouquet-dipoles which is defined below.

\begin{definition}
A graph $G$ is a doubly extended bouquet-dipole if it consists of vertices $u$ and $v$ that are each on a nonempty set of odd cycles, where all other vertices on the cycles have degree two, and $u$ and $v$ are joined by a path of even order that alternates between single even order paths whose internal vertices all have degree two, and multiple even order paths whose internal vertices all have degree two.
\end{definition}

We restate a theorem from Ansill, Jacob, Penzellna, and Savvedra \cite{Ansill} which gives a characterization of all graphs with a failed skew zero forcing number of $0$.

\begin{theorem}{\cite{Ansill}}
$F^{-}(G)=0$ if and only if $G$ is one of the following graphs.
(i) An odd cycle, or a nonempty set of odd cycles whose intersection is a single vertex or (ii) A doubly extended bouquet-dipole.
\end{theorem}

In this paper we provide a characterization of all graphs with a failed skew zero forcing number of $1$. To show a graph has $F^{-}(G)=0$ one has to show that the set where $S=\varnothing $ is a failed zero forcing set and if $S\neq \varnothing $ then all of the vertices in the graph are forced. However to show a graph has $F^{-}(G)=1$ it is more complex. We need to show that there is a vertex which is a stalled set, and that any pair of vertices forces all vertices in the graph.

\section{Results}
\begin{theorem}
The only disconnected graph with $F^-(G)=1$ is $2K_1$.
\end{theorem}
\begin{proof}
Suppose $G$ is disconnected and has more than one two components or a component with more than one vertex. Then let $S$ contain all but one component or the component with two or more vertices. Now $|S|>1$. Now let $G$ have two components that both have one vertex; $G=2K_1$. Let $|S|=1$. It is clear that this is the maximum stalling set of $2K_1$. Therefore, $2K_1$ is the only disconnected graph for which $F^-(G)=1$.
\end{proof}
From here we will assume that all graphs are connected. We first show two small graphs that have a failed skew zero forcing number of $1$.
\begin{lemma}
$F^-(P_3)=1$.
\end{lemma}
\begin{proof}
Let the external vertices be labeled $v_1$ and $v_2$ and the middle vertex be $w$. If $S=\{w\}$, then $S$ is skew stalled; $F^-(P_3)\geq 1$. If $|S|=2$, then without loss of generality there will be a vertex in $S$ with one neighbor outside of $S$, so the whole graph will be skew forced. Therefore, $F^-(P_3)\leq 1$, giving the desired conclusion that $F^-(P_3)=1$.
\end{proof}
\begin{lemma}
$F^-(K_4)=1$.
\end{lemma}
\begin{proof}
Without loss of generality, choose a vertex on $K_4$. Then that vertex has three neighbors outside of $S$ which implies it will not skew force. Then each neighbor outside of $S$ has two neighbors outside of $S$ which implies each neighbor will not skew force. Thus, $F^-(K_4) \geq 1$. Now suppose $|S|=2$. Then the vertices outside of $S$ only have one uncolored neighbor which will be skew forced into $S$, and then any vertex in $S$ will force the remaining vertex that is outside $S$. So, $F^-(K_4) \leq 1$. Therefore, $F^-(K_4)=1$.
\end{proof}
\indent Now we define a substructure and and a lemma that provides a lower bound for the failed skew zero forcing number of a graph.
\begin{definition}
An $n$-blocking is a $P_{2n+1}$ subgraph whose external vertices have degree higher than 2 and internal vertices have degree 2.
\end{definition}
\indent When $n=1$, we will call this substructure a 1-blocking.
\begin{lemma}
(n-Blocking Lemma) If $G$ has no vertices of degree 1 and contains $k$ disjoint blockings of size $n_1,...n_m$, $$F^-(G)\geq \Sigma_{i=1}^m n_i$$.
\end{lemma}
\begin{proof}
Let $P_{2n+1}$ be the largest $n$-blocking on a graph, $G$ with vertex indexing $v_1,...v_{2n+1}$, where $v_1$ and $v_{2n+1}$ are the end vertices with degree larger than 2. Then define $S=\{v_2,v_4,...,v_{2n}\}$ such that $|S|=n$. Then each vertex in $S$ has two neighbors outside $S$ and will not skew force, $v_1$ and $v_{2n+1}$ have degree higher than 2, so they have at least two neighbors outside $S$ and will not skew force, and each $v_{2i+1}$ inside the $n$-blocking has exactly 2 neighbors inside $S$. Since $G$ has no vertices of degree 1, no vertex outside of the $n$-blocking will skew force either. So, $S$ is skew stalled, and $F^-(G)\geq n$.
\end{proof}
\begin{corollary}
If $G$ is an odd cycle with a chord that creates an odd cycle of length $2n+1$, then $F^-(G)=n$
\end{corollary}
\indent Notice that when $n=1$, then this family of graphs is defined by a 1-blocking. This useful result begins our notions of how to characterize all graphs with $F^-(G)=1$. The following three lemmas provide useful descriptions of graphs with $F^-(G)=1$, specifically surrounding their failed skew stalling sets.

It is possible for a graph where $F^-(G)=1$ and there is more than one vertex that is a stalling set. An example of a graph with two different vertices that are both stalling sets is shown in Figure 1.
\begin{center}
  \begin{tikzpicture}
    \coordinate (1) at (0,0);
    \coordinate (2) at (0,2);
    \coordinate (3) at (1.5,1);
    \coordinate (4) at (2.5,1);
    \coordinate (5) at (3.5,0);
    \coordinate (6) at (3.5,1);
    \coordinate (7) at (4.5,1);
    \coordinate (8) at (5.5,1);
    \coordinate (9) at (7,0);
    \coordinate (10) at (7,2);
    \draw (1) -- (2) -- (3) -- (1);
    \draw (3) -- (4) -- (5) -- (6) -- (7) -- (8);
    \draw (8) -- (9) -- (10) -- (8);
    \draw (4) -- (6);
     \foreach \point in {1,2,3,4,6,8,9,10} \fill[white, draw=black] (\point) circle (4pt);
      \foreach \point in {5} \fill (\point) circle (4pt)[blue];
       \foreach \point in {7} \fill (\point) circle (4pt)[red];
         \end{tikzpicture}
         
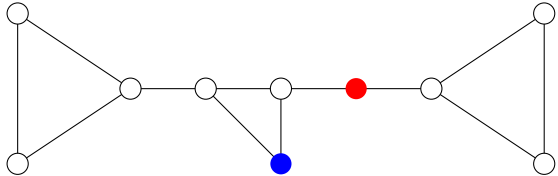
 Figure 1. A graph with two different stalling sets of size 1.
 \end{center}
  In our next lemma we show that it is not possible to have more than 2 different vertices that are each stalling sets.
\begin{lemma}
Unless $G=K_4$, there cannot be more than $2$ vertices that are each failed skew zero forcing sets of size $1$ in a graph where $F^{-}(G)=1$.
\end{lemma}

\begin{proof}
We will assume that $G$ is not $K_4$. We then proceed by contradiction. Let $G$ be a graph and $k\geq 3$ where $F^{-}(G)=1$ and $a_{1},a_{2},...,a_{k}$ are vertices that are each failed skew zero forcing sets of size $1$. We consider two cases. 

\begin{itemize}
    \item Case 1. There does not exist more than two vertices from the set $\{a_{1},a_{2},..,a_{t}\}$ that share a common neighbor. Without loss of generality consider three vertices $a_{1},a_{2}$, and $a_{3}$ where $N(a_{1})\cap N(a_{2})=\emptyset $. Then $\{a_{1},a_{2}\}$ form a failed skew zero forcing set of size $2$, which contradicts the assumption that $F^{-}(G)=1$.
    \item Case 2. There exist vertices $a_{1},a_{2},$ and $a_{3}$ that share a common neighbor $v$. Now since each of the vertices $a_{1},a_{2},$ and $a_{3}$ are failed skew zero forcing sets of size $1$, all of the neighbors of each $a_{i}$, $1\leq i\leq t$ have degree $3$ or more. If $N(a_{1})=N(a_{2})$, $N(a_{1})=N(a_{3})$, and $N(a_{2})=N(a_{3})$ then $F^{-}(G)=\left\vert V(G)\right\vert =n-2.$ Suppose that $N(a_{1})\neq N(a_{2})$, $N(a_{1})\neq N(a_{3})$, and $N(a_{2})\neq N(a_{3})$. Then there is a vertex $x$ where $x\in N(a_{1})$ but $x\notin N(a_{2})$, and there is a vertex $y$ where $y\in N(a_{3})$ but $x\notin N(a_{2})$.  Hence each of the vertices $a_{1},a_{2},$ and $a_{3}$ have two uncolored neighbors each of which have degree at least $3$. Then $\{a_{1},a_{2},a_{3}\}$ forms a failed skew zero forcing set of size $3$, which contradicts the assumption that $F^{-}(G)=1$.
\end{itemize}
\end{proof}

 \begin{lemma}
(Cut vertex lemma) Let $G$ be a graphs with a cut vertex $v$ with $\deg(v)>2$. Let $H_1,H_2,..,H_t$ be the components of $G-v$ where $\left\vert H_{1}\right\vert \geq \left\vert H_{2}\right\vert \geq \cdots \geq \left\vert H_{t}\right\vert $.
Then $F^{-1}(G)\geq 1+\sum_{i=1}^{t-2}\left\vert H_{i}\right\vert $
\end{lemma}
\begin{proof}
If $v$ is a cut vertex of degree $t>2$, we can take the vertices in the $t-2$ largest component(s) and add $v$. Since each vertex in this set either has 0 or 2 uncolored neighbors, this set is a failed zero forcing set. Hence $F^{-1}(G)\geq 1+\sum_{i=1}^{t-2}\left\vert H_{i}\right\vert $
%Suppose $\deg(v)\geq 3$ and $H$ is a subgraph with the least edges to $v$. Then let $S=V(H)\cup\{v\}$. Then $S$ is skew stalled since $v$ is the only vertex with any neighbor outside $S$ and $v$ necessarily has at least two neighbors outside $S$.
\end{proof}
\begin{lemma}
If $F^-(G)=1$ and $\{v\}=S$, $v$ cannot have a neighbor of degree 2
\end{lemma}
\begin{proof}
If $v$ has a neighbor, $w$, of degree 2, then $w$ will skew force its neighbor, $u$, contradicting that $S$ is skew stalled. So, $v$ cannot have a neighbor of degree 2.
\end{proof}

\indent With these lemmas in mind, we can create a necessary condition for all graphs with $F^-(G)=1$.
\begin{lemma}
Excluding $P_3$ and $K_4$, if $F^-(G)=1$, then $G$ contains 1 disjoint 1-blocking and no other disjoint $n$-blocking.
\end{lemma}
\begin{proof}
By the $n$-blocking lemma, since $F^-(G)=1$, if $G$ contains an $n$-blocking, then $\Sigma_{i=0}^m n_i =1$, where $n_i$ is the size of each disjoint $n$-blocking. Hence $G$ contains at most 1 $n$-blocking where $n=1$.\\
\indent To show that $G$ contains at least one disjoint 1-blocking, recall that if $\{v\}$ is our maximum failed skew stalling set, that unless $v$ is on a $P_3$ the neighbors of $v$ must have degree greater than 2. It remains to confirm that $v$ necessarily has degree 2 which will imply the existence of the 1-blocking.\\
\indent Suppose $\deg(v)\geq 3$. We know its neighbors must have degree higher than 2 since if a neighbor $u$ has degree 1 then ${u,v}$ would be a stalled set, contradicting the assumption that $F^{-}(G)=1$. \\\indent Suppose $v$ is a cut vertex of degree $t\geq3$. By Lemma 3.5 the set $S$ consisting of $v$ and vertices in the $t-2$ largest subsets of $G-v$ is skew stalled. \\\indent Next suppose $v$ is not a cut vertex.  Since $v$ has no neighbors of degree 2, $v$ must be adjacent to a vertex $u$, which prevents the creation of a new $n$-blocking. Recall that $\deg(u)\geq 3$.  If $u$ has no neighbor of degree 2, let $S=\{v,u\}$. $S$ is skew stalled with $|S|=2$, contradicting the hypothesis that $F^-(G)=1$. If $u$ has a neighbor of degree 2, that neighbor must either be part of an even path or an odd cycle. If that neighbor is on an even path, let $S=\{v,u\}$ and skew force along the even path until it skew stalls at the next vertex with degree higher than 2. If it is part of an odd cycle and has a neighbor with degree higher than 2, then let $S$ contain only $u$ and the vertices of the cycle to which $u$ belongs. Since $u$ has two uncolored neighbors of degree higher than 2, $S$ is skew stalled with $|S|>1$, which contradicts the original hypothesis. Therefore, if $F^-(G)=1$, then it must be that $\deg(v)=2$.\\
\indent We still have to consider the case when $G$ has two distinct maximum skew stalling sets of cardinality 1, $\{v\}$ and $\{w\}$. By Lemma 3.6, neither vertex can have a neighbor of degree 2. Their shared neighbor, $u$ must have degree 3 or $\{v,w\}$ will be a skew stalling set.\\
\indent Suppose first that $\deg(v)$ and $\deg(w)>2$. If $u$ is a cut vertex, then by the above lemma, $F^-(G)>1$, so $u$ cannot be a cut vertex. If $u$ is not a cut vertex, then its neighbor, $x$ will have degree 2 or higher. If $\deg(x)=2$, then set $S=\{u\}$ and skew force along the even path until $S$ skew stalls upon reaching some vertex with degree 3 or higher. If $\deg(x)>2$ and the neighbor of $x$ has degree 3 or higher, then set $S=\{x,u\}$, which will stall and contradict the original statement of the theorem. So, let $x$ have at least one neighbor of degree 2. Then, since $u$ is not a cut vertex, this neighbor may be on an even path or an odd cycle. If $x$ only has a degree 2 neighbor on an odd cycle, then $x$ must have some other degree 3 neighbor that connects to other parts of the graph, since $u$ is not a cut vertex. In that case, let $S$ contain $x$ and the vertices of that odd cycle. Otherwise, let $S=\{x\}$ and skew force along optional odd cycle and then the even path until it reaches a degree 3 vertex and skew stalls. Then $|S|>1$, contradicting the original hypothesis.
\\\indent Therefore, $G$ contains at most and at least 1 disjoint 1-blocking; $G$ contains exactly 1 disjoint 1 blocking and no other $n$-blocking if $F^-(G)=1$.
\end{proof}
\indent As we close in on the necessary and sufficient conditions for $F^-(G)=1$, we consider the following important substructure that can exist within such graphs:
\begin{definition}
An even multiple path has vertices $u$ and $v$ with 2 or more even paths connecting them and are buffered on both sides by an even path.
\end{definition}
An example of an even multiple path can be seen in the middle of the first graph in Figure 2. Next we provide our main result, giving a full characterization of all graphs with $F^-(G)=1$.
 
    \begin{center}
    \begin{tikzpicture}
    \coordinate (1) at (0,0);
    \coordinate (2) at (-1.25,0.65);
    \coordinate (3) at (-1,2.05);
    \coordinate (4) at (0.35,2.25);
    \coordinate (5) at (1,1);
    \coordinate (6) at (1,2);
    \coordinate (7) at (1.5,2);
    \coordinate (8) at (1,0);
    \coordinate (9) at (1.5,0);
    \coordinate (10) at (1.5,1);
    \coordinate (11) at (2,1);
    \coordinate (12) at (2.5,1);
    \coordinate (13) at (3.5,1);
    \coordinate (14) at (4,1);
    \coordinate (15) at (5,1);
    \coordinate (16) at (5.5,1);
    \coordinate (17) at (6,1);
    \coordinate (18) at (6.5,1);
    \coordinate (19) at (7,1);
    \coordinate (20) at (7.5,1);
    \coordinate (21) at (7.7,1.55);
    \coordinate (22) at (8.25,1.7);
    \coordinate (23) at (8.72,1.39);
    \coordinate (24) at (8.79,0.81);
    \coordinate (25) at (8.37,0.42);
    \coordinate (26) at (7.8,0.5);
    \coordinate (27) at (3,1.5);
    \coordinate (28) at (3.5,1.5);
    \coordinate (29) at (4,1.5);
    \coordinate (30) at (4.5,1.5);
    \coordinate (31) at (3.5,0.5);
    \coordinate (32) at (4,0.5);
    \draw (1) -- (2) -- (3) -- (4) -- (5) -- (1);
    \draw (5) -- (6) -- (7) -- (5);
    \draw (5) -- (8) -- (9) -- (5);
    \draw (5) -- (10) -- (11) -- (12) -- (13) -- (14) -- (15) -- (16) -- (17) -- (18) -- (19) -- (20);
    \draw (20) -- (21) -- (22) -- (23) -- (24) -- (25) -- (26) -- (20);
    \draw (12) -- (27) -- (28) -- (29) -- (30) -- (15);
    \draw (12) -- (31) -- (32) -- (15);
    \draw (2,1.1) arc[start angle=0, end angle=180, x radius = 0.5, y radius=0.15];
    \foreach \point in {1,2,3,4,5,6,7,8,9,11,12,13,14,15,16,17,18,19,20,21,22,23,24,25,26,27,28,29,30,31,32} \fill[white, draw=black] (\point) circle (4pt);
    \foreach \point in {10} \fill[blue] (\point) circle (4pt);
    \end{tikzpicture}
    \end{center}
\begin{center}
        \begin{tikzpicture}
        \coordinate (1) at (0,0);
        \coordinate (2) at (0.75,0);
        \coordinate (3) at (1.5,0);
        \coordinate (4) at (2.25,0);
        \coordinate (5) at (3,0);
        \coordinate (6) at (-0.75,0.6);
        \coordinate (7) at (-1.5,0);
        \coordinate (8) at (-1.15,-0.75);
        \coordinate (9) at (-0.35,-0.75);
        \coordinate (10) at (3.75,0.6);
        \coordinate (11) at (4.5,0);
        \coordinate (12) at (4.15,-0.75);
        \coordinate (13) at (3.35,-0.75);
        \coordinate (14) at (0.375,0.5);
        \coordinate (15) at (1.125,0.5);
        \draw (1)--(2)--(3)--(4)--(5);
        \draw (1)--(6)--(7)--(8)--(9)--(1);
        \draw (5)--(10)--(11)--(12)--(13)--(5);
        \draw (1)--(14)--(15)--(3);
        \foreach \point in {1,3,5,6,7,8,9,10,11,12,13,14,15}, \fill[white,draw=black] (\point) circle (3pt);
        \foreach \point in {2} \fill[blue] (\point) circle (3pt);
        \foreach \point in {4}\fill[red] (\point) circle (3pt);
                
        \coordinate (16) at (5,0.75);
        \coordinate (17) at (5,-0.75);
        \coordinate (18) at (5.5,0);
        \coordinate (19) at (6,0);
        \coordinate (20) at (6.5,0);
        \coordinate (21) at (7,0);
        \coordinate (22) at (7.5,0);
        \coordinate (23) at (8,0);
        \coordinate (24) at (8.5,0.75);
        \coordinate (25) at (8.5,-0.75);
        \draw (16)--(17)--(18)--(16);
        \draw (18)--(19)--(20)--(21)--(22)--(23);
        \draw (23)--(24)--(25)--(23);
        \foreach \point in {16,17,18,20,22,23,24,25} \fill[white,draw=black] (\point) circle (3pt);
        \foreach \point in {19} \fill[blue] (\point) circle (3pt);
        \foreach \point in {21} \fill[red] (\point) circle (3pt);
        \draw (7.5,0.1) arc[start angle = 0, end angle = 180, x radius = 0.5, y radius = 0.25];
        \end{tikzpicture}
        \end{center}
        \begin{center}
    \begin{tikzpicture}
    \coordinate (1) at (0,0);
    \coordinate (2) at (0.75,0);
    \coordinate (3) at (1.36,0.44);
    \coordinate (4) at (1.59,1.158);
    \coordinate (5) at (1.36,1.875);
    \coordinate (6) at (0.75,2.31);
    \coordinate (7) at (0,2.31);
    \coordinate (8) at (-0.61,1.875);
    \coordinate (9) at (-0.84,1.158);
    \coordinate (10) at (-0.61,0.44);
    \coordinate (11) at (0.37,1.157);
    \coordinate (12) at (2.59,1.158);
    \coordinate (13) at (2.59,2.158);
    \draw (1) -- (2) -- (3) -- (4) -- (5) -- (6) -- (7) -- (8) -- (9) -- (10) -- (1);
    \draw (1)--(11)--(6);
    \draw (4)--(12)--(13)--(4);
    \foreach \point in {1,2,3,4,6,7,8,9,10,12,13} \fill[white,draw=black] (\point) circle (3pt);
    \foreach \point in {11} \fill[blue] (\point) circle (3pt);
    \foreach \point in {5} \fill[red] (\point) circle (3pt);
    
    \coordinate (14) at (5,0);
    \coordinate (15) at (6,0);
    \coordinate (16) at (6.8,0.6);
    \coordinate (17) at (7.1,1.55);
    \coordinate (18) at (6.8,2.5);
    \coordinate (19) at (6,3.1);
    \coordinate (20) at (5,3.1);
    \coordinate (21) at (4.2,2.5);
    \coordinate (22) at (3.9,1.55);
    \coordinate (23) at (4.2,0.6);
    \coordinate (24) at (4.7,1.55);
    \coordinate (25) at (5.5,1.55);
    \coordinate (26) at (6.3,1.55);
    \coordinate (27) at (5.5,1.05);
    \coordinate (28) at (6,0.75);
    \coordinate (29) at (5,0.75);
    \coordinate (30) at (7.6,1.75);
    \coordinate (31) at (7.6,2.25);
    \coordinate (32) at (7.6,1.35);
    \coordinate (33) at (7.6,0.85);
    \coordinate (34) at (3.75,2.5);
    \coordinate (35) at (4.4,3.2);
    \coordinate (36) at (5,3.4);
    \coordinate (37) at (6,3.4);
    \coordinate (38) at (6.6,3.2);
    \coordinate (39) at (7.3,2.5);
    \draw (14)--(15)--(16)--(17)--(18)--(19)--(20)--(21)--(22)--(23)--(14);
    \draw (22)--(24)--(25)--(26)--(17);
    \draw (25)--(27)--(28)--(29)--(27);
    \draw (17)--(30)--(31)--(17);
    \draw (17)--(32)--(33)--(17);
    \draw (22)--(34)--(35)--(36)--(37)--(38)--(39)--(17);
    \foreach \point in {14,15,16,17,18,19,20,21,22,23,25,27,28,29,30,31,32,33,34,35,36,37,38,39} \fill[white,draw=black] (\point) circle (3pt);
    \foreach \point in {24} \fill[blue] (\point) circle (3pt);
    \foreach \point in {26} \fill[red] (\point) circle (3pt);
    
    \end{tikzpicture}
    \newline
    \newline
    \medskip
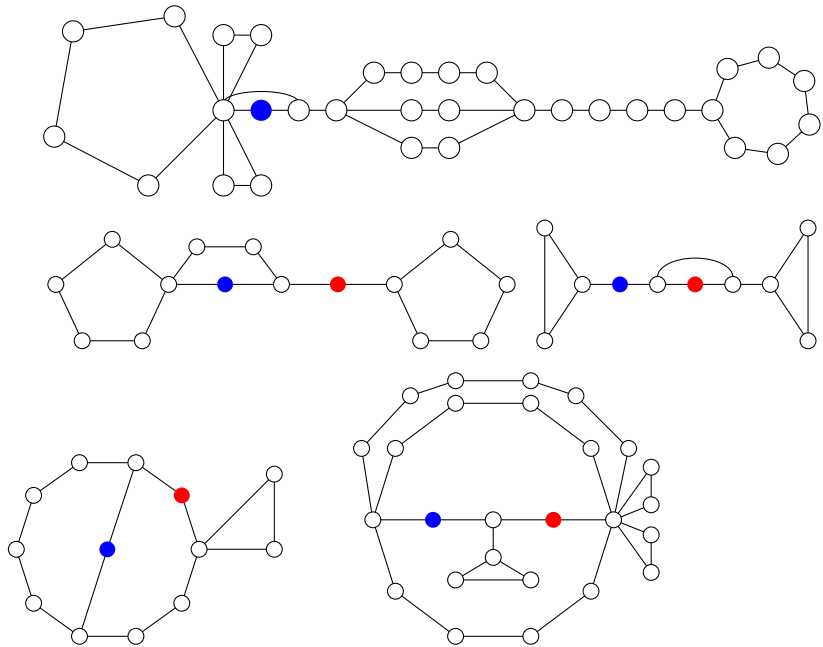
    Figure 2. Some graphs with $F^-(G)=1$
    \end{center}

\begin{theorem}
$F^-(G)=1$ if and only if $G$ is $K_4$ or $P_3$ or $G$ has:
\begin{itemize}
    \item Exactly 1 disjoint 1-blocking, no more than 2 non-disjoint 1-blockings that share a vertex of degree 3, and no other $n$-blocking
    \item All degree 2 vertices are in a 1-blocking, an even path with external vertices with higher degree than 2, and even multiple path, or an odd cycle with exactly one vertex of degree higher than 2
    \item There are no pendant even cycles on $G$ and no vertices of degree 1.
    \item Outside of the 1-blocking, no vertex has more than one neighbor of degree 3.
    \item All but one of the exterior vertices of the 1-blockings can have one neighbor that is in an odd cycle or adjacent to another exterior of a 1-blocking through an even path.
\end{itemize}
\end{theorem}
\begin{proof}
We know $F^-(G)=1$ for $K_4 $ and $P_3$ by Lemmas 3.1 and 3.2.  Suppose $F^-(G)=1$. By Lemma 3.7, if $F^-(G)=1$, then $G$ contains 1 disjoint 1-blocking and no other disjoint $n$-blocking, satisfying the first criterion. Furthermore, suppose there is a degree 2 vertex $v$ in $V(G)$ that is not in one of the specified structures of the second criterion. Then $v$ would be in a separate $n$-blocking, contradicting the first criterion. $G$ cannot have a pendant even cycle since this contradicts the second criteria. If $G$ has a degree 1 vertex, it will skew force its neighbor. If that neighbor is not an element of a maximum failed skew forcing set, this contradicts that $F^-(G)=1$. If it is, then this vertex is not in a 1-blocking. For completeness, suppose that this vertex is a maximum skew stalling set. Therefore it cannot have a neighbor of degree 2. By Lemma 3.5, this would cause $F^-(G)>1$, contradicting that $F^-(G)=1$. Now suppose a vertex outside the 1-blocking has 2 neighbors of degree higher than 2. Then, this vertex will stall in $S$ but color any neighbors of degree 2. These neighbors are either in an odd cycle or an even path. If they are in an odd cycle, they are all skew forced, but will not be able to spread skew forcing to the rest of the graph, as they only have one vertex of degree higher than 2. If they are on an even path, then skew forcing will stall before the next vertex of degree higher than 2. So, the fourth criterion holds. Finally, if all exterior vertices of 1-blockings have a neighbor that is not in an odd cycle or on an even path connected to another exterior vertex of a 1-blocking. Then if each exterior vertex of the 1-blocking is in $S$, they will skew force along their neighbors of degree 2, but, since these vertices are in neither an odd cycle nor a 1-blocking, they must be on an even path where skew forcing will stall before the next vertex of degree higher than 3 that is not an exterior vertex of the 1-blocking. Therefore, the theorem holds that if $F^-(G)=1$, it must have all of the criteria.

Suppose $G$ satisfies all the conditions. Let $v$ be the first vertex we place in $S$. Begin with $v$ being the vertex of degree 2 in the 1-blocking. Each neighbor has degree higher than 2. Therefore, $S$ is skew stalled and $F^-(G)\geq 1$.
Now, there are 5 cases for the different vertices we can start with for $v$.
\begin{itemize}
\item (Case 1) Suppose $v$ is the vertex on a nonempty collection of odd cycles with its degree higher than 2. This will skew force the whole collection of odd cycles into $S$. Then it will follow one of the following cases:
\begin{itemize}
    \item Case 1(a): Let $w$ be a degree 2 neighbor of $v$ in $S$.
    \begin{itemize}
        \item Case 1(a)i: $w$ is on a $P_{2n}$ where it will continue to skew force along the path until the next vertex with degree higher than 2 is skew forced into $S$. Then refer to case 2.
        \item Case 1(a)ii: $w$ is the vertex of degree 2 in a 1-blocking. It will then skew force its neighbor with degree higher than 3, which is the intersection point of a nonempty collection of odd cycles which will be forced or the beginning of one or more even paths, which leads to case 1(a).
    \end{itemize}
    \item Case 1(b): Let $w$ be a degree 3 neighbor of $v$.
    \begin{itemize}
        \item Case 1(b)i: $w$ is the beginning of one or more even paths. Then skew forcing will continue along each path; this case naturally leads to case 1(a).
        \item Case 1(b)ii: $w$ is a vertex in the 1-blocking with degree higher than 2. The vertex of degree 2 in the 1-blocking will skew force its neighbor which has degree higher than 2. Then, that vertex is the intersection of a nonempty collection of odd cycles which will now be skew forced, or the beginning of one or more even path(s), which leads to case 1(a).
    \end{itemize}
\end{itemize}
\indent Since $G$ does not have any pendant even cycles and no vertex with degree 1, eventually the skew forcing will reach some vertex, $z$, that is the beginning of a pendant. $z$ cannot be cases 1(a) or 1(b) since it then would imply the existence of a degree 1 vertex. Furthermore, $z$ cannot be the vertex in case 2(a) since that would create one or more pendant even cycle(s). So, $z$ is the beginning of a pendant of a nonempty collection of odd cycles. Besides the degree 2 vertex of a 1-blocking, all neighbors of $z$ will be in $S$. Then all degree 2 neighbors of $z$ will skew force their neighbors and $z$ will skew force that vertex in a 1-blocking if they are neighbors. Now, each vertex in $S$ has at most one neighbor outside of $S$, so they will skew force their remaining neighbors. Finally, any degree 3 vertex with an edge incident to a nonempty set of odd cycles will skew force the connecting degree 3 vertex and thus skew force the entire set.
\item (Case 2) Now suppose $v$ is a vertex of degree 2 on a cycle. Skew forcing will continue along the cycle until you hit the only vertex with degree higher than 2 on the cycle. At this point, the reader can refer back to the case where that vertex with degree higher than 2 is $v$.
\item (Case 3) Suppose $v$ is the exterior point of a 1-blocking. Automatically, the other exterior vertex of the 1-blocking is skew forced. If $v$ shares a common exterior of a 1-blocking with that vertex, that exterior vertex will also be skew forced. So, all exterior vertices of 1-blockings are in $S$, causing each even path between them to be skew forced. Since there is one exterior vertex with no neighbors besides vertices in an odd cycle or vertices that are connected via an even path to another exterior of a 1-blocking, that exterior vertex will now skew force the degree 2 vertex of the 1-blocking. If any other exterior vertex is adjacent to an even path or multiple even paths that do not lead to another exterior of a 1-blocking, then it will skew force two consecutive vertices along this path, that will end in a nonempty collection of odd cycles or a multiple path, which would lead back to Cases 1 or 2.
\item (Case 4) Suppose $v$ has degree 2 on an even path. It will then skew force along the path until one of the degree 3 vertices mentioned above is skew forced.
\item (Case 5) Suppose there are two non-disjoint 1-blockings. Let $x$ and $y$ be the degree 2 vertices in these 1-blockings. Then if $S=\{x,y\}$, their common vertex has degree 3 and therefore will skew force another vertex, leading to one of the above cases.
\end{itemize}
Since selecting any vertex other than the degree 2 vertex in a 1-blocking will cause the whole graph to become skew forced, and selecting two degree 2 vertices from a 1-blocking causes the whole graph to become skew forced, we know that $F^-(G)\leq 1$. Therefore, $F^-(G)=1$.
\end{proof}
\indent From this result, we observe the following corollary:
\begin{corollary}
If $F^-(G)<2$, $G$ is planar. 
\end{corollary}

\section{Conclusion}
Now that the characterization of graphs with $F^-(G)=0$ and $F^-(G)=1$ have been completed the next logical step is to look at graphs where $F^-(G)=k$ for $k\geq2$. We also proved that a graph with $F^-(G)=1$ can have at most two vertices that can be failed skew zero forcing sets of size 1. An analogous question can be asked. How many different failed skew zero forcing sets of size $t$ exist for a graph with $F^-(G)=t$? 
\newline
\medskip

\textbf{Acknowledgements} This research was supported by the National Science Foundation Research for Undergraduates Award 1950189. 
\printbibliography
\end{document}